\documentclass[]{article}
\usepackage[top=1in,bottom=1in,left=1 in,right=.8in,footskip=0.4in]{geometry}
\usepackage{amsmath, amsthm, amssymb,mathtools,enumerate}
\usepackage{authblk}
\usepackage[round]{natbib}
\usepackage{color,soul}
\usepackage{hyperref}

\usepackage{tikz}
\usepackage{tikz-cd}
\usetikzlibrary{graphs}
\usetikzlibrary{graphs.standard}
\usetikzlibrary{arrows.meta}
\usepackage{calc}
\usepackage{float}
\usetikzlibrary{calc}
\tikzstyle{vertex}=[circle, draw, fill=black, inner sep=0pt, minimum size=6pt]
\tikzstyle{smlv}=[circle, draw, fill=black, inner sep=0pt, minimum size=2pt]
\tikzstyle{grvert}=[circle, draw, gray, fill=gray,  inner sep=0pt, minimum size=6pt]

\newcommand{\smlv}{\node[smlv]}

\newtheorem{theorem}{Theorem}[section]

\def\M{\mathcal{M}}

\title{Corrigendum to ``Topology of matching complexes of complete graphs via discrete Morse theory'' [Discrete Math.\ Theor.\ Comput.\ Sci.\ 26:3\#13 (2024)]}

\author{Anupam Mondal\thanks{\texttt{anupam.mondal@tcgcrest.org}}}
\author{Sajal Mukherjee\thanks{\texttt{sajal.mukherjee@tcgcrest.org}}}
\author{Kuldeep Saha\thanks{\texttt{kuldeep.saha@tcgcrest.org}}}

\affil{\small Institute for Advancing Intelligence (IAI), TCG CREST, Kolkata, India}
\affil{\small Academy of Scientific \& Innovative Research (AcSIR), Ghaziabad, India}

\date{}

\begin{document}

\maketitle

\begin{abstract}
In this corrigendum, we justify that the motivation behind Conjecture~5.1 is based on a mistaken notion. Moreover, we prove a stronger theorem that disproves the conjecture, and obtain an optimal gradient vector field on the matching complex of the complete graph of order 7.

\end{abstract}

\section{On the motivation behind Conjecture~5.1}
In the article ``Topology of matching complexes of complete graphs via discrete Morse theory'', which appeared in Discrete Mathematics \& Theoretical Computer Science, vol.~26:3 (2024), we first constructed a gradient vector field $\M$ on $M_n$ (the matching complex of the complete graph of order $n$) for $n \ge 5$ that is optimal up to a certain dimension. We then extended $\M$ to a more efficient one, viz., $\M^*$, for $M_7$ in particular, in order to compute its homology groups in an efficient combinatorial manner. Once this goal was achieved, we set out to augment $\M^*$ further towards optimality. We applied the technique of \emph{cancellation of a critical pair} by \cite{forman,forman02} \cite[Theorem~2.7]{mms} on $\M^*$ simultaneously twice (by cancelling the critical pairs $\eta_1$ \& $\sigma_4$ and $\eta_2$ \& $\sigma_3$), and ended up with a new gradient vector field $\M^{**}$ \cite[proof of Theorem~1.4]{mms}. With respect to $\M^{**}$, there are 22 critical $2$-simplices, two critical $1$-simplices (viz., $\sigma_1$ and $\sigma_2$), and one critical $0$-simplex, whereas the theoretical lower-bounds for these numbers are 21, 1, 1, respectively.

In the conclusion of that article \cite[Section~5]{mms}, we claimed that $\M^{**}$ cannot be augmented further using the cancellation technique, and conjectured (Conjecture~5.1) that $\M^{**}$ is indeed an optimal gradient vector field on $M_7$. The claim is based on the following premise:
``$\ldots$ there is an $\M^*$-path from a face of a critical (with respect to $\M^*$) $2$-simplex, say $\eta$, to the critical $1$-simplex $\sigma_1$ if and only if there is an $\M^*$-path from a face of $\eta$ to $\sigma_4$.'', which is demonstrably a mistaken notion. For example, let us consider the following critical (with respect to $\M^{*}$) $2$-simplex, say $\eta_3$:
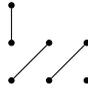
\[\begin{tikzpicture}		
	\smlv (v11) at (0,0) {};
	\smlv (v12) at (0.5,0) {};
	\smlv (v13) at (1,0) {};

	\smlv (v21) at (0,0.5) {};
	\smlv (v22) at (0.5,0.5) {};
	\smlv (v23) at (1,0.5) {};

	\smlv (v31) at (0,1) {};
	
	\path
	(v21) edge (v31)
	(v11) edge (v22)
	(v12) edge (v23)
	;
\end{tikzpicture}\]
We note that there are exactly three gradient paths, starting from a face of $\eta_3$ and ending at a critical $1$-simplex \cite[Fig.~15, Fig.~18, Fig.~20]{mms} as shown in Figure~\ref{eta3} below. 
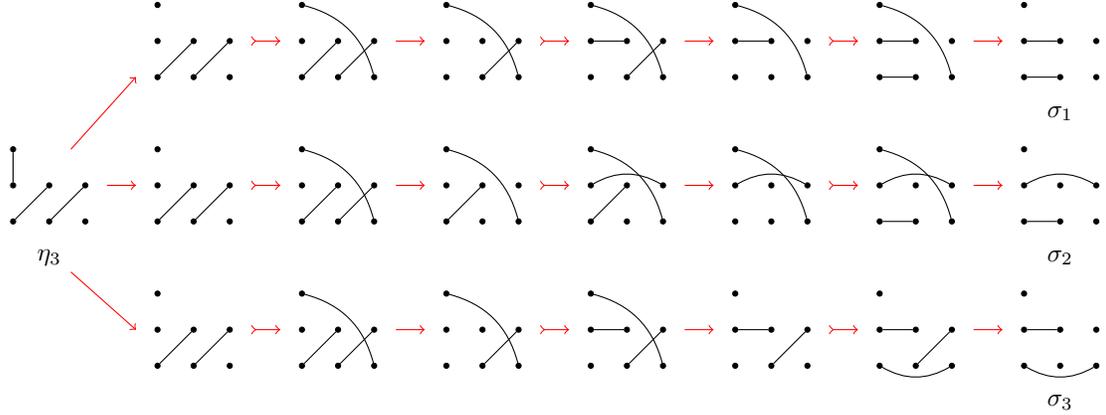
\begin{figure}[!ht]
	\centering
	\begin{tikzpicture}
		\begin{scope}[scale=0.96, transform shape]
			\begin{scope}[xshift=-2cm]
				\node at (0.5,-0.5) {$\eta_3$};
				
				\smlv (v11) at (0,0) {};
				\smlv (v12) at (0.5,0) {};
				\smlv (v13) at (1,0) {};

				\smlv (v21) at (0,0.5) {};
				\smlv (v22) at (0.5,0.5) {};
				\smlv (v23) at (1,0.5) {};

				\smlv (v31) at (0,1) {};
				
				\path
				(v21) edge (v31)
				(v11) edge (v22)
				(v12) edge (v23)
				;
			\end{scope}
			
			\draw [-to,red](-1.2,1) -- (-0.3,2);
			\begin{scope}[yshift=2cm]
				\begin{scope}
					
					\smlv (v11) at (0,0) {};
					\smlv (v12) at (0.5,0) {};
					\smlv (v13) at (1,0) {};

					\smlv (v21) at (0,0.5) {};
					\smlv (v22) at (0.5,0.5) {};
					\smlv (v23) at (1,0.5) {};

					\smlv (v31) at (0,1) {};
					
					\path
					(v11) edge (v22)
					(v12) edge (v23)
					;
				\end{scope}
				
				\draw [to reversed-to,red](1.3,0.5) -- (1.7,0.5);
				\begin{scope}[xshift=2cm]			
					\smlv (v11) at (0,0) {};
					\smlv (v12) at (0.5,0) {};
					\smlv (v13) at (1,0) {};

					\smlv (v21) at (0,0.5) {};
					\smlv (v22) at (0.5,0.5) {};
					\smlv (v23) at (1,0.5) {};

					\smlv (v31) at (0,1) {};
					
					\path
					(v11) edge (v22)
					(v12) edge (v23)
					;
					\path[bend right]
					(v13) edge (v31)
					;
				\end{scope}
				
				\draw [-to,red](3.3,0.5) -- (3.7,0.5);
				\begin{scope}[xshift=4cm]			
					\smlv (v11) at (0,0) {};
					\smlv (v12) at (0.5,0) {};
					\smlv (v13) at (1,0) {};

					\smlv (v21) at (0,0.5) {};
					\smlv (v22) at (0.5,0.5) {};
					\smlv (v23) at (1,0.5) {};

					\smlv (v31) at (0,1) {};
					
					\path
					(v12) edge (v23)
					;
					\path[bend right]
					(v13) edge (v31)
					;
				\end{scope}
				
				\draw [to reversed-to,red](5.3,0.5) -- (5.7,0.5);
				\begin{scope}[xshift=6cm]			
					\smlv (v11) at (0,0) {};
					\smlv (v12) at (0.5,0) {};
					\smlv (v13) at (1,0) {};

					\smlv (v21) at (0,0.5) {};
					\smlv (v22) at (0.5,0.5) {};
					\smlv (v23) at (1,0.5) {};

					\smlv (v31) at (0,1) {};
					
					\path
					(v12) edge (v23)
					(v21) edge (v22)
					;
					\path[bend right]
					(v13) edge (v31)
					;
				\end{scope}
				
				\draw [-to,red](7.3,0.5) -- (7.7,0.5);
				\begin{scope}[xshift=8cm]			
					\smlv (v11) at (0,0) {};
					\smlv (v12) at (0.5,0) {};
					\smlv (v13) at (1,0) {};

					\smlv (v21) at (0,0.5) {};
					\smlv (v22) at (0.5,0.5) {};
					\smlv (v23) at (1,0.5) {};

					\smlv (v31) at (0,1) {};
					
					\path
					(v21) edge (v22)
					;
					\path[bend right]
					(v13) edge (v31)
					;
				\end{scope}
				
				\draw [to reversed-to,red](9.3,0.5) -- (9.7,0.5);
				\begin{scope}[xshift=10cm]			
					\smlv (v11) at (0,0) {};
					\smlv (v12) at (0.5,0) {};
					\smlv (v13) at (1,0) {};

					\smlv (v21) at (0,0.5) {};
					\smlv (v22) at (0.5,0.5) {};
					\smlv (v23) at (1,0.5) {};

					\smlv (v31) at (0,1) {};
					
					\path
					(v11) edge (v12)
					(v21) edge (v22)
					;
					\path[bend right]
					(v13) edge (v31)
					;
				\end{scope}
				
				\draw [-to,red](11.3,0.5) -- (11.7,0.5);
				\begin{scope}[xshift=12cm]
					\node at (0.5,-0.5) {$\sigma_1$};
					
					\smlv (v11) at (0,0) {};
					\smlv (v12) at (0.5,0) {};
					\smlv (v13) at (1,0) {};

					\smlv (v21) at (0,0.5) {};
					\smlv (v22) at (0.5,0.5) {};
					\smlv (v23) at (1,0.5) {};

					\smlv (v31) at (0,1) {};
					
					\path
					(v11) edge (v12)
					(v21) edge (v22)
					;
				\end{scope}
			\end{scope}

			\draw [-to,red](-0.7,0.5) -- (-0.3,0.5);
			\begin{scope}
				\begin{scope}
					
					\smlv (v11) at (0,0) {};
					\smlv (v12) at (0.5,0) {};
					\smlv (v13) at (1,0) {};

					\smlv (v21) at (0,0.5) {};
					\smlv (v22) at (0.5,0.5) {};
					\smlv (v23) at (1,0.5) {};

					\smlv (v31) at (0,1) {};
					
					\path
					(v11) edge (v22)
					(v12) edge (v23)
					;
				\end{scope}
				
				\draw [to reversed-to,red](1.3,0.5) -- (1.7,0.5);
				\begin{scope}[xshift=2cm]			
					\smlv (v11) at (0,0) {};
					\smlv (v12) at (0.5,0) {};
					\smlv (v13) at (1,0) {};

					\smlv (v21) at (0,0.5) {};
					\smlv (v22) at (0.5,0.5) {};
					\smlv (v23) at (1,0.5) {};

					\smlv (v31) at (0,1) {};
					
					\path
					(v11) edge (v22)
					(v12) edge (v23)
					;
					\path[bend right]
					(v13) edge (v31)
					;
				\end{scope}
				
				\draw [-to,red](3.3,0.5) -- (3.7,0.5);
				\begin{scope}[xshift=4cm]			
					\smlv (v11) at (0,0) {};
					\smlv (v12) at (0.5,0) {};
					\smlv (v13) at (1,0) {};

					\smlv (v21) at (0,0.5) {};
					\smlv (v22) at (0.5,0.5) {};
					\smlv (v23) at (1,0.5) {};

					\smlv (v31) at (0,1) {};
					
					\path
					(v11) edge (v22)
					;
					\path[bend right]
					(v13) edge (v31)
					;
				\end{scope}
				
				\draw [to reversed-to,red](5.3,0.5) -- (5.7,0.5);
				\begin{scope}[xshift=6cm]			
					\smlv (v11) at (0,0) {};
					\smlv (v12) at (0.5,0) {};
					\smlv (v13) at (1,0) {};

					\smlv (v21) at (0,0.5) {};
					\smlv (v22) at (0.5,0.5) {};
					\smlv (v23) at (1,0.5) {};

					\smlv (v31) at (0,1) {};
					
					\path
					(v11) edge (v22)
					;
					\path[bend right]
					(v13) edge (v31)
					(v23) edge (v21)
					;
				\end{scope}
				
				\draw [-to,red](7.3,0.5) -- (7.7,0.5);
				\begin{scope}[xshift=8cm]			
					\smlv (v11) at (0,0) {};
					\smlv (v12) at (0.5,0) {};
					\smlv (v13) at (1,0) {};

					\smlv (v21) at (0,0.5) {};
					\smlv (v22) at (0.5,0.5) {};
					\smlv (v23) at (1,0.5) {};

					\smlv (v31) at (0,1) {};
					
					\path[bend right]
					(v13) edge (v31)
					(v23) edge (v21)
					;
				\end{scope}
				
				\draw [to reversed-to,red](9.3,0.5) -- (9.7,0.5);
				\begin{scope}[xshift=10cm]			
					\smlv (v11) at (0,0) {};
					\smlv (v12) at (0.5,0) {};
					\smlv (v13) at (1,0) {};

					\smlv (v21) at (0,0.5) {};
					\smlv (v22) at (0.5,0.5) {};
					\smlv (v23) at (1,0.5) {};

					\smlv (v31) at (0,1) {};
					
					\path
					(v11) edge (v12)
					;
					\path[bend right]
					(v13) edge (v31)
					(v23) edge (v21)
					;
				\end{scope}
				
				\draw [-to,red](11.3,0.5) -- (11.7,0.5);
				\begin{scope}[xshift=12cm]
					\node at (0.5,-0.5) {$\sigma_2$};
					
					\smlv (v11) at (0,0) {};
					\smlv (v12) at (0.5,0) {};
					\smlv (v13) at (1,0) {};

					\smlv (v21) at (0,0.5) {};
					\smlv (v22) at (0.5,0.5) {};
					\smlv (v23) at (1,0.5) {};

					\smlv (v31) at (0,1) {};
					
					\path
					(v11) edge (v12)
					;
					\path[bend right]
					(v23) edge (v21)
					;
				\end{scope}
			\end{scope}
			
			\draw [-to,red](-1.2,-0.7) -- (-0.3,-1.5);
			\begin{scope}[yshift=-2cm]
				\begin{scope}
					
					\smlv (v11) at (0,0) {};
					\smlv (v12) at (0.5,0) {};
					\smlv (v13) at (1,0) {};

					\smlv (v21) at (0,0.5) {};
					\smlv (v22) at (0.5,0.5) {};
					\smlv (v23) at (1,0.5) {};

					\smlv (v31) at (0,1) {};
					
					\path
					(v11) edge (v22)
					(v12) edge (v23)
					;
				\end{scope}
				
				\draw [to reversed-to,red](1.3,0.5) -- (1.7,0.5);
				\begin{scope}[xshift=2cm]			
					\smlv (v11) at (0,0) {};
					\smlv (v12) at (0.5,0) {};
					\smlv (v13) at (1,0) {};

					\smlv (v21) at (0,0.5) {};
					\smlv (v22) at (0.5,0.5) {};
					\smlv (v23) at (1,0.5) {};

					\smlv (v31) at (0,1) {};
					
					\path
					(v11) edge (v22)
					(v12) edge (v23)
					;
					\path[bend right]
					(v13) edge (v31)
					;
				\end{scope}
				
				\draw [-to,red](3.3,0.5) -- (3.7,0.5);
				\begin{scope}[xshift=4cm]			
					\smlv (v11) at (0,0) {};
					\smlv (v12) at (0.5,0) {};
					\smlv (v13) at (1,0) {};

					\smlv (v21) at (0,0.5) {};
					\smlv (v22) at (0.5,0.5) {};
					\smlv (v23) at (1,0.5) {};

					\smlv (v31) at (0,1) {};
					
					\path
					(v12) edge (v23)
					;
					\path[bend right]
					(v13) edge (v31)
					;
				\end{scope}
				
				\draw [to reversed-to,red](5.3,0.5) -- (5.7,0.5);
				\begin{scope}[xshift=6cm]			
					\smlv (v11) at (0,0) {};
					\smlv (v12) at (0.5,0) {};
					\smlv (v13) at (1,0) {};

					\smlv (v21) at (0,0.5) {};
					\smlv (v22) at (0.5,0.5) {};
					\smlv (v23) at (1,0.5) {};

					\smlv (v31) at (0,1) {};
					
					\path
					(v12) edge (v23)
					(v21) edge (v22)
					;
					\path[bend right]
					(v13) edge (v31)
					;
				\end{scope}
				
				\draw [-to,red](7.3,0.5) -- (7.7,0.5);
				\begin{scope}[xshift=8cm]			
					\smlv (v11) at (0,0) {};
					\smlv (v12) at (0.5,0) {};
					\smlv (v13) at (1,0) {};

					\smlv (v21) at (0,0.5) {};
					\smlv (v22) at (0.5,0.5) {};
					\smlv (v23) at (1,0.5) {};

					\smlv (v31) at (0,1) {};
					
					\path
					(v12) edge (v23)
					(v21) edge (v22)
					;
				\end{scope}
				
				\draw [to reversed-to,red](9.3,0.5) -- (9.7,0.5);
				\begin{scope}[xshift=10cm]			
					\smlv (v11) at (0,0) {};
					\smlv (v12) at (0.5,0) {};
					\smlv (v13) at (1,0) {};

					\smlv (v21) at (0,0.5) {};
					\smlv (v22) at (0.5,0.5) {};
					\smlv (v23) at (1,0.5) {};

					\smlv (v31) at (0,1) {};
					
					\path
					(v12) edge (v23)
					(v21) edge (v22)
					;
					\path[bend left]
					(v13) edge (v11)
					;
				\end{scope}
				
				\draw [-to,red](11.3,0.5) -- (11.7,0.5);
				\begin{scope}[xshift=12cm]		
					\node at (0.5,-0.5) {$\sigma_3$};
					
					\smlv (v11) at (0,0) {};
					\smlv (v12) at (0.5,0) {};
					\smlv (v13) at (1,0) {};

					\smlv (v21) at (0,0.5) {};
					\smlv (v22) at (0.5,0.5) {};
					\smlv (v23) at (1,0.5) {};

					\smlv (v31) at (0,1) {};
					
					\path
					(v21) edge (v22)
					;
					\path[bend left]
					(v13) edge (v11)
					;
				\end{scope}
				
			\end{scope}
			
		\end{scope}
	\end{tikzpicture}
	\caption{All three possible $\M^*$-paths that start from a $1$-simplex contained in $\eta_3$, and end at a critical $1$-simplex.}\label{eta3}
\end{figure}
These paths end at $\sigma_1$, $\sigma_2$, and $\sigma_3$, respectively. However, there is no gradient path from a face of $\eta_3$ to $\sigma_4$, disproving the claim.

\section{Obtaining an optimal gradient vector field on $M_7$}
We now prove that $\M^{**}$ indeed can be augmented further, and as a consequence, we end up with an optimal gradient vector field on $M_7$. We present the following stronger version of a previously established theorem \cite[Theorem~1.4]{mms}.

\begin{theorem}
	There is a gradient vector field on $M_7$ such that there are 21 critical $2$-simplices, one critical $1$-simplex, and one critical $0$-simplex. Consequently, it is an optimal gradient vector field on $M_7$.
\end{theorem}

\begin{proof}
	Let us consider the following three critical (with respect to $\M^{*}$) $2$-simplices.
		
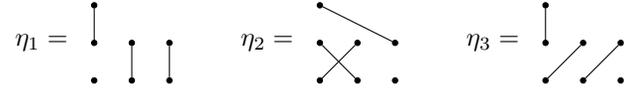
\[\begin{tikzpicture}
		\begin{scope}
			\node at (-0.7,0.5) {$\eta_1 = $};
			
			\smlv (v11) at (0,0) {};
			\smlv (v12) at (0.5,0) {};
			\smlv (v13) at (1,0) {};

			\smlv (v21) at (0,0.5) {};
			\smlv (v22) at (0.5,0.5) {};
			\smlv (v23) at (1,0.5) {};

			\smlv (v31) at (0,1) {};
			
			\path
			(v12) edge (v22)
			(v13) edge (v23)
			(v21) edge (v31)
			;
		\end{scope}
		
		\begin{scope}[xshift=3cm]
			\node at (-0.7,0.5) {$\eta_2 = $};
			
			\smlv (v11) at (0,0) {};
			\smlv (v12) at (0.5,0) {};
			\smlv (v13) at (1,0) {};

			\smlv (v21) at (0,0.5) {};
			\smlv (v22) at (0.5,0.5) {};
			\smlv (v23) at (1,0.5) {};

			\smlv (v31) at (0,1) {};
			
			\path
			(v11) edge (v22)
			(v12) edge (v21)
			(v23) edge (v31)
			;
		\end{scope}
		
		\begin{scope}[xshift=6cm]
			\node at (-0.7,0.5) {$\eta_3 = $};
			
			\smlv (v11) at (0,0) {};
			\smlv (v12) at (0.5,0) {};
			\smlv (v13) at (1,0) {};

			\smlv (v21) at (0,0.5) {};
			\smlv (v22) at (0.5,0.5) {};
			\smlv (v23) at (1,0.5) {};

			\smlv (v31) at (0,1) {};
			
			\path
			(v21) edge (v31)
			(v11) edge (v22)
			(v12) edge (v23)
			;
		\end{scope}
	\end{tikzpicture}\]
	
	We note that there are exactly two gradient paths, starting from a face of $\eta_1$ and ending at a critical $1$-simplex \cite[Fig.~11, Fig.~14, Fig.~23]{mms} as shown in Figure~\ref{eta1} below.	
	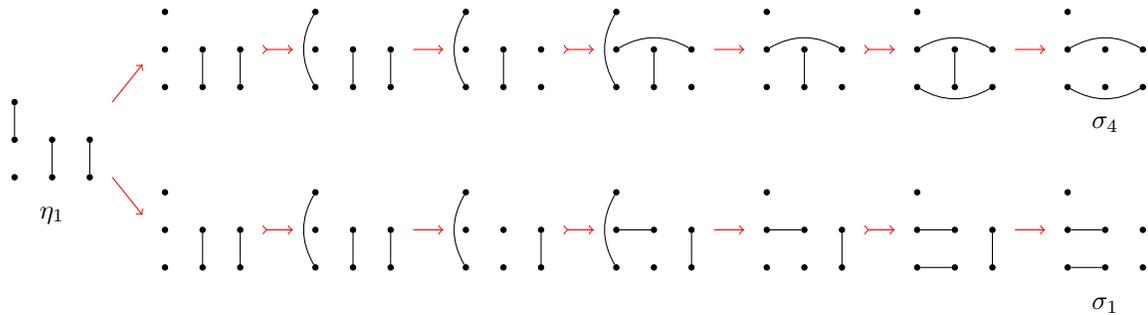
\begin{figure}[!ht]
		\centering
		\begin{tikzpicture}
			\begin{scope}[xshift=-2cm]
				\node at (0.5,-0.5) {$\eta_1$};
				
				\smlv (v11) at (0,0) {};
				\smlv (v12) at (0.5,0) {};
				\smlv (v13) at (1,0) {};

				\smlv (v21) at (0,0.5) {};
				\smlv (v22) at (0.5,0.5) {};
				\smlv (v23) at (1,0.5) {};

				\smlv (v31) at (0,1) {};
				
				\path
				(v12) edge (v22)
				(v13) edge (v23)
				(v21) edge (v31)
				;
			\end{scope}
			
			\draw [-to,red](-0.7,1) -- (-0.3,1.5);
			\begin{scope}[yshift=1.2cm]
				\begin{scope}
					
					\smlv (v11) at (0,0) {};
					\smlv (v12) at (0.5,0) {};
					\smlv (v13) at (1,0) {};

					\smlv (v21) at (0,0.5) {};
					\smlv (v22) at (0.5,0.5) {};
					\smlv (v23) at (1,0.5) {};

					\smlv (v31) at (0,1) {};
					
					\path
					(v12) edge (v22)
					(v13) edge (v23)
					;
				\end{scope}
				
				\draw [to reversed-to,red](1.3,0.5) -- (1.7,0.5);
				\begin{scope}[xshift=2cm]	
					\smlv (v11) at (0,0) {};
					\smlv (v12) at (0.5,0) {};
					\smlv (v13) at (1,0) {};

					\smlv (v21) at (0,0.5) {};
					\smlv (v22) at (0.5,0.5) {};
					\smlv (v23) at (1,0.5) {};

					\smlv (v31) at (0,1) {};
					
					\path
					(v12) edge (v22)
					(v13) edge (v23)
					;
					\path[bend left]
					(v11) edge (v31)
					;
				\end{scope}
				
				\draw [-to,red](3.3,0.5) -- (3.7,0.5);
				\begin{scope}[xshift=4cm]		
					\smlv (v11) at (0,0) {};
					\smlv (v12) at (0.5,0) {};
					\smlv (v13) at (1,0) {};

					\smlv (v21) at (0,0.5) {};
					\smlv (v22) at (0.5,0.5) {};
					\smlv (v23) at (1,0.5) {};

					\smlv (v31) at (0,1) {};
					
					\path
					(v12) edge (v22)
					;
					\path[bend left]
					(v11) edge (v31)
					;
				\end{scope}
				
				\draw [to reversed-to,red](5.3,0.5) -- (5.7,0.5);
				\begin{scope}[xshift=6cm]		
					\smlv (v11) at (0,0) {};
					\smlv (v12) at (0.5,0) {};
					\smlv (v13) at (1,0) {};

					\smlv (v21) at (0,0.5) {};
					\smlv (v22) at (0.5,0.5) {};
					\smlv (v23) at (1,0.5) {};

					\smlv (v31) at (0,1) {};
					
					\path
					(v12) edge (v22)
					;
					\path[bend left]
					(v11) edge (v31)
					(v21) edge (v23)
					;
				\end{scope}
				
				\draw [-to,red](7.3,0.5) -- (7.7,0.5);
				\begin{scope}[xshift=8cm]	
					\smlv (v11) at (0,0) {};
					\smlv (v12) at (0.5,0) {};
					\smlv (v13) at (1,0) {};

					\smlv (v21) at (0,0.5) {};
					\smlv (v22) at (0.5,0.5) {};
					\smlv (v23) at (1,0.5) {};

					\smlv (v31) at (0,1) {};
					
					\path
					(v12) edge (v22)
					;
					\path[bend left]
					(v21) edge (v23)
					;
				\end{scope}
				
				\draw [to reversed-to,red](9.3,0.5) -- (9.7,0.5);
				\begin{scope}[xshift=10cm]			
					\smlv (v11) at (0,0) {};
					\smlv (v12) at (0.5,0) {};
					\smlv (v13) at (1,0) {};

					\smlv (v21) at (0,0.5) {};
					\smlv (v22) at (0.5,0.5) {};
					\smlv (v23) at (1,0.5) {};

					\smlv (v31) at (0,1) {};
					
					\path
					(v12) edge (v22)
					;
					\path[bend left]
					(v13) edge (v11)
					(v21) edge (v23)
					;
				\end{scope}
				
				\draw [-to,red](11.3,0.5) -- (11.7,0.5);
				\begin{scope}[xshift=12cm]
					\node at (0.5,-0.5) {$\sigma_4$};
					
					\smlv (v11) at (0,0) {};
					\smlv (v12) at (0.5,0) {};
					\smlv (v13) at (1,0) {};

					\smlv (v21) at (0,0.5) {};
					\smlv (v22) at (0.5,0.5) {};
					\smlv (v23) at (1,0.5) {};

					\smlv (v31) at (0,1) {};
					
					\path[bend left]
					(v13) edge (v11)
					(v21) edge (v23)
					;
				\end{scope}
			\end{scope}
			
			\draw [-to,red](-0.7,0) -- (-0.3,-0.5);
			\begin{scope}[yshift=-1.2cm]
				\begin{scope}
					
					\smlv (v11) at (0,0) {};
					\smlv (v12) at (0.5,0) {};
					\smlv (v13) at (1,0) {};

					\smlv (v21) at (0,0.5) {};
					\smlv (v22) at (0.5,0.5) {};
					\smlv (v23) at (1,0.5) {};

					\smlv (v31) at (0,1) {};
					
					\path
					(v12) edge (v22)
					(v13) edge (v23)
					;
				\end{scope}
				
				\draw [to reversed-to,red](1.3,0.5) -- (1.7,0.5);
				\begin{scope}[xshift=2cm]		
					\smlv (v11) at (0,0) {};
					\smlv (v12) at (0.5,0) {};
					\smlv (v13) at (1,0) {};

					\smlv (v21) at (0,0.5) {};
					\smlv (v22) at (0.5,0.5) {};
					\smlv (v23) at (1,0.5) {};

					\smlv (v31) at (0,1) {};
					
					\path
					(v12) edge (v22)
					(v13) edge (v23)
					;
					\path[bend left]
					(v11) edge (v31)
					;
				\end{scope}
				
				\draw [-to,red](3.3,0.5) -- (3.7,0.5);
				\begin{scope}[xshift=4cm]		
					\smlv (v11) at (0,0) {};
					\smlv (v12) at (0.5,0) {};
					\smlv (v13) at (1,0) {};

					\smlv (v21) at (0,0.5) {};
					\smlv (v22) at (0.5,0.5) {};
					\smlv (v23) at (1,0.5) {};

					\smlv (v31) at (0,1) {};
					
					\path
					(v13) edge (v23)
					;
					\path[bend left]
					(v11) edge (v31)
					;
				\end{scope}
				
				\draw [to reversed-to,red](5.3,0.5) -- (5.7,0.5);
				\begin{scope}[xshift=6cm]		
					\smlv (v11) at (0,0) {};
					\smlv (v12) at (0.5,0) {};
					\smlv (v13) at (1,0) {};

					\smlv (v21) at (0,0.5) {};
					\smlv (v22) at (0.5,0.5) {};
					\smlv (v23) at (1,0.5) {};

					\smlv (v31) at (0,1) {};
					
					\path
					(v13) edge (v23)
					(v21) edge (v22)
					;
					\path[bend left]
					(v11) edge (v31)
					;
				\end{scope}
				
				\draw [-to,red](7.3,0.5) -- (7.7,0.5);
				\begin{scope}[xshift=8cm]	
					\smlv (v11) at (0,0) {};
					\smlv (v12) at (0.5,0) {};
					\smlv (v13) at (1,0) {};

					\smlv (v21) at (0,0.5) {};
					\smlv (v22) at (0.5,0.5) {};
					\smlv (v23) at (1,0.5) {};

					\smlv (v31) at (0,1) {};
					
					\path
					(v13) edge (v23)
					(v21) edge (v22)
					;
				\end{scope}
				
				\draw [to reversed-to,red](9.3,0.5) -- (9.7,0.5);
				\begin{scope}[xshift=10cm]		
					\smlv (v11) at (0,0) {};
					\smlv (v12) at (0.5,0) {};
					\smlv (v13) at (1,0) {};

					\smlv (v21) at (0,0.5) {};
					\smlv (v22) at (0.5,0.5) {};
					\smlv (v23) at (1,0.5) {};

					\smlv (v31) at (0,1) {};
					
					\path
					(v11) edge (v12)
					(v13) edge (v23)
					(v21) edge (v22)
					;
				\end{scope}
				
				\draw [-to,red](11.3,0.5) -- (11.7,0.5);
				\begin{scope}[xshift=12cm]		
					\node at (0.5,-0.5) {$\sigma_1$};
					
					\smlv (v11) at (0,0) {};
					\smlv (v12) at (0.5,0) {};
					\smlv (v13) at (1,0) {};

					\smlv (v21) at (0,0.5) {};
					\smlv (v22) at (0.5,0.5) {};
					\smlv (v23) at (1,0.5) {};

					\smlv (v31) at (0,1) {};
					
					\path
					(v11) edge (v12)
					(v21) edge (v22)
					;
				\end{scope}
			\end{scope}
		\end{tikzpicture}
		\caption{Only two possible $\M^*$-paths that start from a $1$-simplex contained in $\eta_1$, and end at a critical $1$-simplex.}\label{eta1}
	\end{figure}
	
	Also, there are exactly two gradient paths from a face of $\eta_2$ to a critical $1$-simplex \cite[Fig.~18, Fig.~19, Fig.~25]{mms} as shown in Figure~\ref{eta2} below.	
	\begin{figure}[!ht]
		\centering
		\begin{tikzpicture}
			\begin{scope}[scale=0.96, transform shape]
				\begin{scope}[xshift=-2cm]
					\node at (0.5,-0.5) {$\eta_2$};
					
					\smlv (v11) at (0,0) {};
					\smlv (v12) at (0.5,0) {};
					\smlv (v13) at (1,0) {};

					\smlv (v21) at (0,0.5) {};
					\smlv (v22) at (0.5,0.5) {};
					\smlv (v23) at (1,0.5) {};

					\smlv (v31) at (0,1) {};
					
					\path
					(v11) edge (v22)
					(v12) edge (v21)
					(v23) edge (v31)
					;
				\end{scope}
				
				\draw [-to,red](-0.7,1) -- (-0.3,1.5);
				\begin{scope}[yshift=1.2cm]
					\begin{scope}
						
						\smlv (v11) at (0,0) {};
						\smlv (v12) at (0.5,0) {};
						\smlv (v13) at (1,0) {};

						\smlv (v21) at (0,0.5) {};
						\smlv (v22) at (0.5,0.5) {};
						\smlv (v23) at (1,0.5) {};

						\smlv (v31) at (0,1) {};
						
						\path
						(v11) edge (v22)
						(v12) edge (v21)
						;
					\end{scope}
					
					\draw [to reversed-to,red](1.3,0.5) -- (1.7,0.5);
					\begin{scope}[xshift=2cm]			
						\smlv (v11) at (0,0) {};
						\smlv (v12) at (0.5,0) {};
						\smlv (v13) at (1,0) {};

						\smlv (v21) at (0,0.5) {};
						\smlv (v22) at (0.5,0.5) {};
						\smlv (v23) at (1,0.5) {};

						\smlv (v31) at (0,1) {};
						
						\path
						(v11) edge (v22)
						(v12) edge (v21)
						;
						\path[bend right]
						(v13) edge (v31)
						;
					\end{scope}
					
					\draw [-to,red](3.3,0.5) -- (3.7,0.5);
					\begin{scope}[xshift=4cm]			
						\smlv (v11) at (0,0) {};
						\smlv (v12) at (0.5,0) {};
						\smlv (v13) at (1,0) {};

						\smlv (v21) at (0,0.5) {};
						\smlv (v22) at (0.5,0.5) {};
						\smlv (v23) at (1,0.5) {};

						\smlv (v31) at (0,1) {};
						
						\path
						(v11) edge (v22)
						;
						\path[bend right]
						(v13) edge (v31)
						;
					\end{scope}
					
					\draw [to reversed-to,red](5.3,0.5) -- (5.7,0.5);
					\begin{scope}[xshift=6cm]			
						\smlv (v11) at (0,0) {};
						\smlv (v12) at (0.5,0) {};
						\smlv (v13) at (1,0) {};

						\smlv (v21) at (0,0.5) {};
						\smlv (v22) at (0.5,0.5) {};
						\smlv (v23) at (1,0.5) {};

						\smlv (v31) at (0,1) {};
						
						\path
						(v11) edge (v22)
						;
						\path[bend right]
						(v13) edge (v31)
						(v23) edge (v21)
						;
					\end{scope}
					
					\draw [-to,red](7.3,0.5) -- (7.7,0.5);
					\begin{scope}[xshift=8cm]			
						\smlv (v11) at (0,0) {};
						\smlv (v12) at (0.5,0) {};
						\smlv (v13) at (1,0) {};

						\smlv (v21) at (0,0.5) {};
						\smlv (v22) at (0.5,0.5) {};
						\smlv (v23) at (1,0.5) {};

						\smlv (v31) at (0,1) {};
						
						\path[bend right]
						(v13) edge (v31)
						(v23) edge (v21)
						;
					\end{scope}
					
					\draw [to reversed-to,red](9.3,0.5) -- (9.7,0.5);
					\begin{scope}[xshift=10cm]			
						\smlv (v11) at (0,0) {};
						\smlv (v12) at (0.5,0) {};
						\smlv (v13) at (1,0) {};

						\smlv (v21) at (0,0.5) {};
						\smlv (v22) at (0.5,0.5) {};
						\smlv (v23) at (1,0.5) {};

						\smlv (v31) at (0,1) {};
						
						\path
						(v11) edge (v12)
						;
						\path[bend right]
						(v13) edge (v31)
						(v23) edge (v21)
						;
					\end{scope}
					
					\draw [-to,red](11.3,0.5) -- (11.7,0.5);
					\begin{scope}[xshift=12cm]
						\node at (0.5,-0.5) {$\sigma_2$};
						
						\smlv (v11) at (0,0) {};
						\smlv (v12) at (0.5,0) {};
						\smlv (v13) at (1,0) {};

						\smlv (v21) at (0,0.5) {};
						\smlv (v22) at (0.5,0.5) {};
						\smlv (v23) at (1,0.5) {};

						\smlv (v31) at (0,1) {};
						
						\path
						(v11) edge (v12)
						;
						\path[bend left]
						(v21) edge (v23)
						;
					\end{scope}
				\end{scope}
				
				\draw [-to,red](-0.7,0) -- (-0.3,-0.5);
				\begin{scope}[yshift=-1.2cm]
					\begin{scope}
						
						\smlv (v11) at (0,0) {};
						\smlv (v12) at (0.5,0) {};
						\smlv (v13) at (1,0) {};

						\smlv (v21) at (0,0.5) {};
						\smlv (v22) at (0.5,0.5) {};
						\smlv (v23) at (1,0.5) {};

						\smlv (v31) at (0,1) {};
						
						\path
						(v12) edge (v21)
						(v23) edge (v31)
						;
					\end{scope}
					
					\draw [to reversed-to,red](1.3,0.5) -- (1.7,0.5);
					\begin{scope}[xshift=2cm]			
						\smlv (v11) at (0,0) {};
						\smlv (v12) at (0.5,0) {};
						\smlv (v13) at (1,0) {};

						\smlv (v21) at (0,0.5) {};
						\smlv (v22) at (0.5,0.5) {};
						\smlv (v23) at (1,0.5) {};

						\smlv (v31) at (0,1) {};
						
						\path
						(v12) edge (v21)
						(v23) edge (v31)
						;
						\path[bend right]
						(v11) edge (v13)
						;
					\end{scope}
					
					\draw [-to,red](3.3,0.5) -- (3.7,0.5);
					\begin{scope}[xshift=4cm]			
						\smlv (v11) at (0,0) {};
						\smlv (v12) at (0.5,0) {};
						\smlv (v13) at (1,0) {};

						\smlv (v21) at (0,0.5) {};
						\smlv (v22) at (0.5,0.5) {};
						\smlv (v23) at (1,0.5) {};

						\smlv (v31) at (0,1) {};
						
						\path
						(v23) edge (v31)
						;
						\path[bend right]
						(v11) edge (v13)
						;
					\end{scope}
					
					\draw [to reversed-to,red](5.3,0.5) -- (5.7,0.5);
					\begin{scope}[xshift=6cm]			
						\smlv (v11) at (0,0) {};
						\smlv (v12) at (0.5,0) {};
						\smlv (v13) at (1,0) {};

						\smlv (v21) at (0,0.5) {};
						\smlv (v22) at (0.5,0.5) {};
						\smlv (v23) at (1,0.5) {};

						\smlv (v31) at (0,1) {};
						
						\path
						(v21) edge (v22)
						(v23) edge (v31)
						;
						\path[bend right]
						(v11) edge (v13)
						;
					\end{scope}
					
					\draw [-to,red](7.3,0.5) -- (7.7,0.5);
					\begin{scope}[xshift=8cm]	
						\node at (0.5,-0.5) {$\sigma_3$};
						
						\smlv (v11) at (0,0) {};
						\smlv (v12) at (0.5,0) {};
						\smlv (v13) at (1,0) {};

						\smlv (v21) at (0,0.5) {};
						\smlv (v22) at (0.5,0.5) {};
						\smlv (v23) at (1,0.5) {};

						\smlv (v31) at (0,1) {};
						
						\path
						(v21) edge (v22)
						;
						\path[bend right]
						(v11) edge (v13)
						;
					\end{scope}
				\end{scope}
			\end{scope}
		\end{tikzpicture}
		\caption{Only two possible $\M^*$-paths that start from a $1$-simplex contained in $\eta_2$, and end at a critical $1$-simplex.}\label{eta2}
	\end{figure}
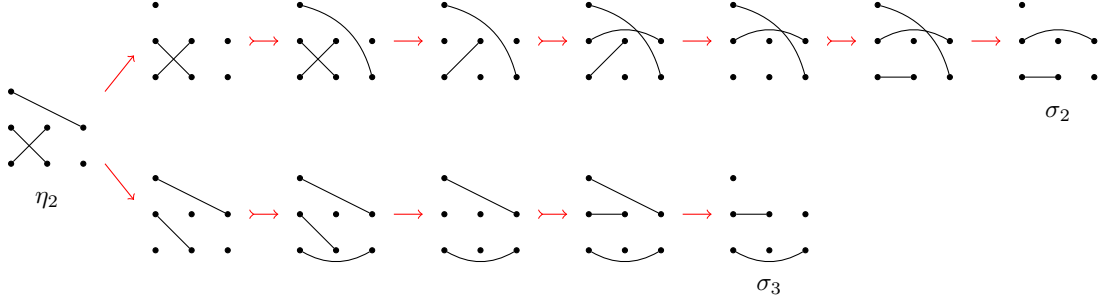

	From Figure~\ref{eta3}, Figure~\ref{eta1}, and Figure~\ref{eta2}, we observe that there is  
	\begin{enumerate}[(i)]
		\item a unique $\M^*$-path that starts from a face of the critical $2$-simplex $\eta_1$, and ends at the critical $1$-simplex $\sigma_4$,
		\item a unique $\M^*$-path that starts from a face of the critical $2$-simplex $\eta_2$, and ends at the critical $1$-simplex $\sigma_3$, and 
		\item a unique $\M^*$-path that starts from a face of the critical $2$-simplex $\eta_3$, and ends at the critical $1$-simplex $\sigma_1$.
	\end{enumerate}

	Moreover, we also observe that, other than $\eta_1 \mapsto \sigma_4$, $\eta_2 \mapsto \sigma_3$, $\eta_3 \mapsto \sigma_1$, there is no other bijection $\pi : \{\eta_1, \eta_2, \eta_3\} \to \{\sigma_1, \sigma_3, \sigma_4\}$, such that there is a $\M^*$-path from a face of $\eta_i$ to $\pi(\eta_i)$ for $i \in \{1,2,3\}$. Thus, a criterion by \cite{hersh} \cite[Theorem~2.8]{mms} for reversing multiple gradient paths to simultaneously cancel several pairs of critical simplices is satisfied. This allows us to apply the aforementioned technique of cancellation of a critical pair by reversing gradient paths thrice for the cancellable pairs $\eta_1$ \& $\sigma_4$, $\eta_2$ \& $\sigma_3$, and $\eta_3$ \& $\sigma_1$, and end up with a gradient vector field, say $\M^\odot$, from $\M^*$.
	
	With respect to $\M^\odot$, $\eta_1$, $\eta_2$, $\eta_3$, $\sigma_4$, $\sigma_3$, and  $\sigma_1$ are no longer critical, while the criticality of all other simplices remains unchanged. Thus, there are $24-3=21$ critical $2$-simplices, $4-3=1$ critical $1$-simplex (viz., $\sigma_4$), and one critical $0$-simplex, with respect to $\M^\odot$. Since, these numbers attain the theoretical lower-bounds, we conclude that $\M^\odot$ is an optimal gradient vector field on $M_7$.
	\end{proof}


\begin{thebibliography}{30}
	\providecommand{\natexlab}[1]{#1}
	\providecommand{\url}[1]{\texttt{#1}}
	\expandafter\ifx\csname urlstyle\endcsname\relax
	\providecommand{\doi}[1]{doi: #1}\else
	\providecommand{\doi}{doi: \begingroup \urlstyle{rm}\Url}\fi

	
	\bibitem[Forman(1998)]{forman}
	R.~Forman.
	\newblock {M}orse theory for cell complexes.
	\newblock \emph{Advances in Mathematics}, 134\penalty0 (1):\penalty0 90--145,
	1998.
	\newblock \doi{10.1006/aima.1997.1650}.
	\newblock URL
	\url{https://www.sciencedirect.com/science/article/pii/S0001870897916509}.
	
	\bibitem[Forman(2002)]{forman02}
	R.~Forman.
	\newblock A user's guide to discrete {M}orse theory.
	\newblock \emph{S{\'e}minaire Lotharingien de Combinatoire [electronic only]},
	48:\penalty0 B48c--35, 2002.
	
	\bibitem[Hersh(2005)]{hersh}
	P.~Hersh.
	\newblock On optimizing discrete {M}orse functions.
	\newblock \emph{Advances in Applied Mathematics}, 35\penalty0 (3):\penalty0
	294--322, 2005.
	\newblock \doi{10.1016/j.aam.2005.04.001}.
	\newblock URL
	\url{https://www.sciencedirect.com/science/article/pii/S0196885805000370}.
	
	
	\bibitem[Mondal et~al.(2024)Mondal, Mukherjee, and Saha]{mms}
	A. Mondal, S. Mukherjee, and K. Saha.
	\newblock Topology of matching complexes of complete graphs via discrete {M}orse theory.
	\newblock \emph{Discrete Mathematics \& Theoretical Computer Science}, 26:3~\#13, 2024.
	\newblock URL \url{https://doi.org/10.46298/dmtcs.12887}.
	
\end{thebibliography}
\end{document}